\documentclass[11pt,a4paper]{article}
\usepackage[latin1]{inputenc}
\usepackage{amsmath}
\usepackage{amsthm}
\usepackage{amsfonts}
\usepackage{amsfonts,amsthm,latexsym,amsmath,amssymb,amscd,epsfig,psfrag,enumerate}
\usepackage{graphics,graphicx, bezier, float, color, hyperref}
\usepackage{amssymb,url}
\usepackage{multienum}
\usepackage[table]{xcolor}
\usepackage{multicol,multirow}
\usepackage{graphicx}
\usepackage{fancyvrb}
\usepackage{parskip}
\usepackage[toc,page]{appendix}
\usepackage[top=2.7cm, bottom=2.7cm, left=1.5cm, right=1.5cm]{geometry}
\usepackage{xcolor}

\usepackage{blkarray}
\newtheorem{theorem}{Theorem}[section]
\newtheorem{lemma}{Lemma}[section]

\usepackage[none]{hyphenat}[section]
\newtheorem{definition}{Definition}[section]

\numberwithin{equation}{section}
\numberwithin{table}{section}
\numberwithin{figure}{section}

\title{On $k$-Pell numbers that are Palindromes formed by two distinct Repdigits}
\author{Herbert Batte$^{1} $ and Darius Guma$^{2,*}$}
\date{}

\begin{document}
	\maketitle
	\abstract{ Let $k \ge 2$ and consider the sequence $\{P_n^{(k)}\}_{n \ge 2-k}$ of $k$-generalized Pell numbers, which begins with the first $k$ terms as $0, \ldots, 0, 0, 1$, and satisfies the recurrence relation
	$	P_n^{(k)} = 2P_{n-1}^{(k)} + P_{n-2}^{(k)} + \cdots + P_{n-k}^{(k)}$ for all  $n \ge 2$. In this work, we identify all terms in the $k$-Pell sequence that can be expressed as palindromes formed by concatenating two distinct repdigits.
	} 
	
	{\bf Keywords and phrases}: $k$-generalized Pell numbers; linear forms in logarithms; LLL-algorithm, 2-adic valuation.
	
	{\bf 2020 Mathematics Subject Classification}: 11B39, 11D61, 11D45.
	
	\thanks{$ ^{*} $ Corresponding author}
	
\section{Introduction}
\subsection{Background}

For any integer $k \geq 2$, the sequence of $k$-generalized Pell numbers, denoted by $\{P_n^{(k)}\}_{n\in\mathbb{Z}}$, is defined recursively by
\[
P_n^{(k)} = 2P_{n-1}^{(k)} + P_{n-2}^{(k)} + \cdots + P_{n-k}^{(k)} \quad \text{for all } n \geq 2,
\]
with initial terms \( P_{2-k}^{(k)} = \cdots = P_{-1}^{(k)} = P_{0}^{(k)} = 0 \) and $P_1^{(k)} = 1$. This family of sequences extends the classical Pell sequence, which corresponds to the case $k = 2$ and has been studied extensively in number theory (see, for example, \cite{Herera}). As an example, the first few values for $P_n^{(k)}$, for some small values of $k$, are given in Table \ref{tab} below.
\begin{table}[h!]\label{tab}
	\centering
	\caption{First non-zero $k$--Pell numbers}
	\begin{tabular}{|c|c|l|}
		\hline
		$k$ & Name & First non-zero terms \\
		\hline
		2 & Pell   & 1, 2, 5, 12, 29, 70, 169, 408, 985, 2378, 5741, 13860, 33461, $\ldots$ \\
		3 & 3-Pell & 1, 2, 5, 13, 33, 84, 214, 545, 1388, 3535, 9003, 22929, 58396, $\ldots$ \\
		4 & 4-Pell & 1, 2, 5, 13, 34, 88, 228, 591, 1532, 3971, 10293, 26680, 69156, $\ldots$ \\
		5 & 5-Pell & 1, 2, 5, 13, 34, 89, 232, 605, 1578, 4116, 10736, 28003, 73041, $\ldots$ \\
		6 & 6-Pell & 1, 2, 5, 13, 34, 89, 233, 609, 1592, 4162, 10881, 28447, 74371, $\ldots$ \\
		7 & 7-Pell & 1, 2, 5, 13, 34, 89, 233, 610, 1596, 4176, 10927, 28592, 74815, $\ldots$ \\
		8 & 8-Pell & 1, 2, 5, 13, 34, 89, 233, 610, 1597, 4180, 10941, 28638, 74960, $\ldots$ \\
		9 & 9-Pell & 1, 2, 5, 13, 34, 89, 233, 610, 1597, 4181, 10945, 28652, 75006, $\ldots$ \\
		10 & 10-Pell & 1, 2, 5, 13, 34, 89, 233, 610, 1597, 4181, 10946, 28656, 75020, $\ldots$ \\
		\hline
	\end{tabular}
\end{table}

A \textit{repdigit} in base 10 is a number composed entirely of repeated digits. Such a number can be represented as
\[
N = \underbrace{\overline{d \ldots d}}_{\ell \text{ times}} = d \left( \frac{10^\ell - 1}{9} \right),
\]
where $0 \leq d \leq 9$ and $\ell \geq 1$. Investigations into the Diophantine nature of recurrence sequences, especially concerning their interaction with repdigits, have garnered considerable interest in recent years. Some works have examined whether certain sequence elements can be represented as sums or concatenations of special kinds of numbers.

Initial work by Luca and Banks \cite{banks} provided preliminary insights into such questions, though the scope was limited. In a more specific case, \cite{BraHer} studied $k$-generalized Pell numbers expressible as repdigits, showing that for this instance, we only have $P_{5}^{(3)} = 33$ and $P_{6}^{(4)} = 88$. This was as a result of extending the work in \cite{faye}.

Recent studies have also focused on palindromic numbers formed from repdigits within linear recurrence sequences. A palindrome is an integer that reads identically forward and backward -- for instance 33, 7007, 1220221, or 959. Research into palindromes occurring in recurrence sequences has intensified. Chalebgwa and Ddamulira \cite{chal} found that 151 and 616 are the only Padovan numbers that can be expressed as palindromic concatenations of two distinct repdigits. Likewise, it was shown in \cite{emong} that 595 is the sole such occurrence in the Narayana's cows sequence. Additional related results can be found in \cite{batte2}, \cite{kaggwa} and \cite{batte}. In \cite{batte}, the authors showed that $F_{11}^{(5)}=464$ is the only $k$-Fibonacci number that can be written as palindromic concatenations of two distinct repdigits.

Inspired by this growing body of work, we turn our attention to the $k$-Pell sequence. Specifically, we investigate solutions to the Diophantine equation
\begin{align}\label{eq:main}
	P_n^{(k)} = \overline{\underbrace{d_1 \ldots d_1}_{\ell \text{ times}} \underbrace{d_2 \ldots d_2}_{m \text{ times}} \underbrace{d_1 \ldots d_1}_{\ell \text{ times}}},
\end{align}
where $d_1, d_2 \in \{0,1,\dots,9\}$ with $d_1 > 0$, $d_1 \ne d_2$, and $\ell, m \geq 1$. That is, we search for those $k$--Pell numbers that are palindromes formed by concatenating two distinct repdigits in a symmetric pattern. Since a number must contain at least three digits to meet this structure, we consider $n \geq 7$ throughout the paper. We prove the following.

\subsection{Main Results}\label{sec:1.2l}
\begin{theorem}\label{thm1.1l} 
	The only $k$-Pell numbers satisfying the Diophantine equation \eqref{eq:main} are
	\begin{align*}
		P_{8}^{(3)}=545 \qquad\text{and}\qquad P_{7}^{(5)}=232.
	\end{align*}
\end{theorem}

The approach used here to prove Theorem \ref{thm1.1l} is as follows.
\begin{enumerate}[(a)]
	\item  We begin by transforming Eq.~\eqref{eq:main} into two separate linear forms in logarithms of algebraic numbers, ensuring both are non-zero and relatively small. Then, we apply Matveev's theorem twice to derive lower bounds on these linear forms, which allows us to obtain polynomial bounds on $n$ in terms of the parameter $k$.
	
	\item We proceed by splitting two cases depending on $k$, i.e., $k \leq 1400$ and $k > 1400$. For the case $k \leq 1400$, we utilize a reduction technique attributed to the LLL-algorithm, to further decrease the size of the bounds, making them more manageable computationally. When $k > 1400$, we apply approximations from \cite{BraHer}, noting that the dominant root of $P_n^{(k)}$ is very close to $\varphi^2$, where $$\varphi := \dfrac{1+\sqrt{5}}{2},$$
	is the golden ratio. This allows us to replace the root with $\varphi$ in our logarithmic estimates, thereby obtaining absolute upper bounds for all involved variables. These bounds are then reduced using the LLL-algorithm.
\end{enumerate}

\section{Some identities on $k$-Pell numbers}
We begin with an identity from \cite{kilis}, which relates $P_n^{(k)}$ with the $m$-th Fibonacci number. It states that
\begin{align}\label{eq:2.1}
	P_n^{(k)} =  F_{2n-1},
\end{align}
for all $1 \le n \le k+1$, while the next term is $P_{k+2}^{(k)} =  F_{2k+3}-1$. The characteristic polynomial of $(P_n^{(k)})_{n\in {\mathbb Z}}$ is given by
\[
\Psi_k(x) = x^k - 2x^{k-1} - \cdots - x - 1.
\]
This polynomial is irreducible in $\mathbb{Q}[x]$ and it possesses a unique real root $\alpha:=\alpha(k)>1$. All the other roots of $\Psi_k(x)$ are inside the unit circle,  see \cite{Herera}. The particular root $\alpha$ can be found in the interval
\begin{align}\label{eq2.3m}
	\varphi^2(1 - \varphi^{-k} ) < \alpha < \varphi^2,\qquad \text{for} \qquad k\ge 2,
\end{align}
as noted in \cite{Herera}. As in the classical case when $k=2$, it was shown in \cite{Herera} that 
\begin{align}\label{fn_up}
	\alpha^{n-2} \le P_n^{(k)}\le \alpha^{n-1}, \qquad \text{holds for all}\qquad n\ge1, \quad k\ge 2.
\end{align}

Next, we rewrite relation \eqref{eq:main} as
\begin{align}\label{eq2.3l}
	P_n^{(k)} &= \overline{\underbrace{d_1 \ldots d_1}_{\ell \text{ times}}\underbrace{d_2 \ldots d_2}_{m \text{ times}}\underbrace{d_1 \ldots d_1}_{\ell \text{ times}}}\nonumber\\
	&=\overline{\underbrace{d_1 \ldots d_1}_{\ell \text{ times}}}\cdot 10^{\ell+m}+\overline{\underbrace{d_2 \ldots d_2}_{m \text{ times}}}\cdot 10^{\ell}+\overline{\underbrace{d_1 \ldots d_1}_{\ell \text{ times}}}\nonumber\\
	&= d_1 \left( \frac{10^\ell - 1}{9} \right)\cdot 10^{\ell+m}+d_2 \left( \frac{10^m - 1}{9} \right)\cdot 10^{\ell}+d_1 \left( \frac{10^\ell - 1}{9} \right) \nonumber  \\
	&=\dfrac{1}{9}\left(d_1\cdot 10^{2\ell+m}-(d_1-d_2)\cdot 10^{\ell+m} +(d_1-d_2)\cdot 10^{\ell}-d_1 \right),
\end{align} 
where \( d_1, d_2 \in \{0, 1, 2, \ldots, 9\}\), \( d_1 > 0 \) and \(d_1\ne d_2\). Since we are working under the assumption that $n\ge 7$, then equation \eqref{eq2.3l} together with \eqref{fn_up} imply that
\begin{align*}
	\alpha^{n-1} &\geq P_n^{(k)} = \frac{1}{9} \left( d_1 \cdot 10^{2\ell+m} - (d_1 - d_2) \cdot 10^{\ell+m} + (d_1 - d_2) \cdot 10^{\ell} - d_1 \right) \\
	&\geq \frac{1}{9} \left( 1 \cdot 10^{2\ell+m} - (9 - 0) \cdot 10^{\ell+m} + (1 - 9) \cdot 10^{\ell} - 9 \right) \\
	&= \frac{1}{9} \cdot 10^{2\ell+m} \left( 1 - \frac{9}{10^{\ell}} - \frac{8}{10^{\ell+m}} - \frac{9}{10^{2\ell+m}} \right) \\
	&\geq \frac{1}{9} \cdot 10^{2\ell+m} \left( 1 - \frac{9}{10^1} - \frac{8}{10^2} - \frac{9}{10^3} \right) \\
	&> 10^{2\ell+m-3},
\end{align*}
where we have used the fact that $1 \leq d_1 \leq 9$, $0 \leq d_2 \leq 9$ and  $\ell, m \geq 1$ are positive integers. Taking logarithms on both sides yields 
\begin{align}\label{eq2.4l}
	2\ell+m<n,
\end{align}
holds for all $n\ge 7$. Again, if we combine equation \eqref{eq2.3l} together with \eqref{fn_up}, we get
\begin{align*}
	\alpha^{n-2} &\leq P_n^{(k)} 
	< 10^{2\ell+m},
\end{align*}
for all $\ell$, $m\ge 1$. Moreover, the above relation also follows from the fact that since $P_n^{(k)}$ has $2\ell+m$ digits, then $P_n^{(k)}<10^{2\ell+m}$. Taking logarithms on both sides of the above inequality gives 
\begin{align}\label{eq2.4m}
n<	6(2\ell+m)+2.
\end{align}

Lastly, for $k\ge 2$, we define
\begin{align*}
	f_k(x):=\dfrac{x - 1}{(k + 1)x^2 - 3kx + k - 1} = \dfrac{x - 1}{k(x^2 - 3x + 1) + x^2 - 1}.
\end{align*}
In Lemma 1 of \cite{BraHer}, Bravo and Herera showed that
\begin{align*}
	0.276<f_k(\alpha)<0.5 \qquad\text{and}\qquad |f_k(\alpha_i)|<1,
\end{align*}
holds for all $2\le i\le k$, where $\alpha_i$ for $i=2,\ldots,k$ are the roots of $\Psi_k(x)$ inside the unit circle. This means that $f_k(\alpha)$ is not an algebraic integer. Moreover, it was also shown in Theorem 3.1 of \cite{Herera} that
\begin{align}\label{eq2.6m}
	P_n^{(k)}=\displaystyle\sum_{i=1}^{k}f_k(\alpha_i)\alpha_i^{n}\qquad\text{and}\qquad\left|P_n^{(k)}-f_k(\alpha)\alpha^{n}\right|<\dfrac{1}{2},
\end{align}
hold for all $k\ge 2$ and $n\ge 2-k$. The first part of \eqref{eq2.6m} provides a Binet-type representation for \( P_n^{(k)} \). In addition, the second inequality in \eqref{eq2.6m} indicates that the impact of the roots located inside the unit circle on \( P_n^{(k)} \) is negligible. A better estimate than \eqref{eq2.6m} appears as Lemma 2 in \cite{BraHer}, but with a more restricted range of $n$ in terms of $k$. It states that if $k\ge 30$ and $n>1$, then
\begin{align}\label{pk_b1}
	\left| f_k(\alpha)\alpha^{n}-\dfrac{\varphi^{2n}}{\varphi+2}\right| <  \left(\dfrac{\varphi^{2n}}{\varphi+2}\right)\frac{4}{\varphi^{k/2}},\qquad {\text{\rm provided}}\qquad n<2^{k/2}.
\end{align}

\section{Methods}
\subsection{Linear forms in logarithms}
We apply four instances of lower bounds due to Baker for nonzero linear forms in three logarithms of algebraic numbers. Various results of this kind exist in the literature, such as those by Baker and W{\"u}stholz \cite{BW}, and Matveev \cite{matl}. Prior to stating these inequalities, we introduce the concept of the height of an algebraic number, which we define below.

\begin{definition}\label{def2.1l}
	Let $ \gamma $ be an algebraic number of degree $ d $ with minimal primitive polynomial over the integers $$ a_{0}x^{d}+a_{1}x^{d-1}+\cdots+a_{d}=a_{0}\prod_{i=1}^{d}(x-\gamma^{(i)}), $$ where the leading coefficient $ a_{0} $ is positive. Then, the logarithmic height of $ \gamma$ is given by $$ h(\gamma):= \dfrac{1}{d}\Big(\log a_{0}+\sum_{i=1}^{d}\log \max\{|\gamma^{(i)}|,1\} \Big). $$
\end{definition}
 In particular, if $ \gamma$ is a rational number represented as $\gamma:=p/q$ with coprime integers $p$ and $ q\ge 1$, then $ h(\gamma ) = \log \max\{|p|, q\} $. 
The following properties of the logarithmic height function $ h(\cdot) $ will be used in the rest of the paper without further reference:
\begin{equation}\nonumber
	\begin{aligned}
		h(\gamma_{1}\pm\gamma_{2}) &\leq h(\gamma_{1})+h(\gamma_{2})+\log 2;\\
		h(\gamma_{1}\gamma_{2}^{\pm 1} ) &\leq h(\gamma_{1})+h(\gamma_{2});\\
		h(\gamma^{s}) &= |s|h(\gamma)  \quad {\text{\rm valid for}}\quad s\in \mathbb{Z}.
	\end{aligned}
\end{equation}
With these properties, it was proved in \cite[Lemma 3]{BraHer} that 
\begin{align}\label{eqh}
	h\left(f_k(\alpha)\right)<4k\log \varphi+k\log (k+1), \qquad \text{for all}\qquad k\ge 2.
\end{align}

A linear form in logarithms is an expression
\begin{equation*}
	\Lambda:=b_1\log \gamma_1+\cdots+b_t\log \gamma_t,
\end{equation*}
where for us $\gamma_1,\ldots,\gamma_t$ are positive real  algebraic numbers and $b_1,\ldots,b_t$ are nonzero integers. We assume, $\Lambda\ne 0$. We need lower bounds 
for $|\Lambda|$. We write ${\mathbb K}:={\mathbb Q}(\gamma_1,\ldots,\gamma_t)$ and $D$ for the degree of ${\mathbb K}$.
We start with the general form due to Matveev, see Theorem 9.4 in \cite{matl}. 

\begin{theorem}[Matveev, see Theorem 9.4 in \cite{matl}]
	\label{thm:Matl} 
	Put $\Gamma:=\gamma_1^{b_1}\cdots \gamma_t^{b_t}-1=e^{\Lambda}-1$. Assume $\Gamma\ne 0$. Then 
	$$
	\log |\Gamma|>-1.4\cdot 30^{t+3}\cdot t^{4.5} \cdot D^2 (1+\log D)(1+\log B)A_1\cdots A_t,
	$$
	where $B\ge \max\{|b_1|,\ldots,|b_t|\}$ and $A_i\ge \max\{Dh(\gamma_i),|\log \gamma_i|,0.16\}$ for $i=1,\ldots,t$.
\end{theorem}

\subsection{Reduction methods}
The bounds derived from Matveev's theorem are often too large to be directly useful for explicit computations. To overcome this limitation, we employ a refinement technique involving the LLL-algorithm. Before explaining this approach, we provide some background.

Let $k$ be a positive integer. A subset $\mathcal{L}$ of the $k$-dimensional real vector space $\mathbb{R}^k$ is defined as a \emph{lattice} if there exists a basis $\{b_1, b_2, \ldots, b_k\}$ of $\mathbb{R}^k$ such that
\begin{align*}
	\mathcal{L} = \sum_{i=1}^{k} \mathbb{Z} b_i = \left\{ \sum_{i=1}^{k} r_i b_i \mid r_i \in \mathbb{Z} \right\}.
\end{align*}
We say that $b_1, b_2, \ldots, b_k$ form a basis for $\mathcal{L}$, or that they span $\mathcal{L}$. We
call $k$ the rank of $ \mathcal{L}$. The determinant $\text{det}(\mathcal{L})$, of $\mathcal{L}$ is defined by
\begin{align*}
	\text{det}(\mathcal{L}) = | \det(b_1, b_2, \ldots, b_k) |,
\end{align*}
with the $b_i$'s being written as column vectors. This is a positive real number that does not depend on the choice of the basis (see \cite{Cas}, Section 1.2).

Given linearly independent vectors $b_1, b_2, \ldots, b_k $ in $ \mathbb{R}^k$, we refer back to the Gram--Schmidt orthogonalization technique. This method allows us to inductively define vectors $b^*_i$ (with $1 \leq i \leq k$) and real coefficients $\mu_{i,j}$ (for $1 \leq j \leq i \leq k$). Specifically,
\begin{align*}
	b^*_i &= b_i - \sum_{j=1}^{i-1} \mu_{i,j} b^*_j,~~~
	\mu_{i,j} = \dfrac{\langle b_i, b^*_j\rangle }{\langle b^*_j, b^*_j\rangle},
\end{align*}
where \( \langle \cdot , \cdot \rangle \)  denotes the ordinary inner product on \( \mathbb{R}^k \). Notice that \( b^*_i \) is the orthogonal projection of \( b_i \) on the orthogonal complement of the span of \( b_1, \ldots, b_{i-1} \), and that \( \mathbb{R}b_i \) is orthogonal to the span of \( b^*_1, \ldots, b^*_{i-1} \) for \( 1 \leq i \leq k \). It follows that \( b^*_1, b^*_2, \ldots, b^*_k \) is an orthogonal basis of \( \mathbb{R}^k \). 
\begin{definition}
	The basis $b_1, b_2, \ldots, b_n$ for the lattice $\mathcal{L}$ is called reduced if
	\begin{align*}
		\| \mu_{i,j} \| &\leq \frac{1}{2}, \quad \text{for} \quad 1 \leq j < i \leq n,~~
		\text{and}\\
		\|b^*_{i}+\mu_{i,i-1} b^*_{i-1}\|^2 &\geq \frac{3}{4}\|b^*_{i-1}\|^2, \quad \text{for} \quad 1 < i \leq n,
	\end{align*}
	where $ \| \cdot \| $ denotes the ordinary Euclidean length. The constant $ {3}/{4}$ above is arbitrarily chosen, and may be replaced by any fixed real number $ y $ in the interval ${1}/{4} < y < 1$ {\rm(see \cite{LLL}, Section 1)}.
\end{definition}
Let $\mathcal{L}\subseteq\mathbb{R}^k$ be a $k-$dimensional lattice  with reduced basis $b_1,\ldots,b_k$ and denote by $B$ the matrix with columns $b_1,\ldots,b_k$. 
We define
\[
l\left( \mathcal{L},y\right)= \left\{ \begin{array}{c}
	\min_{x\in \mathcal{L}}||x-y|| \quad  ;\quad y\not\in \mathcal{L}\\
\min_{0\ne x\in \mathcal{L}}||x|| \quad~~  ;\quad y\in \mathcal{L}
\end{array}
\right.,
\]
where $||\cdot||$ denotes the Euclidean norm on $\mathbb{R}^k$. It is well established that the LLL--algorithm can be used to efficiently compute a lower bound for $l\left( \mathcal{L}, y \right)$ in polynomial time. Specifically, it provides a positive constant $c_1$ such that $l\left( \mathcal{L}, y \right) \geq c_1$ (see \cite{SMA}, Section V.4).
\begin{lemma}\label{lem2.5m}
	Let $y\in\mathbb{R}^k$ and $z=B^{-1}y$ with $z=(z_1,\ldots,z_k)^T$. Furthermore, 
	\begin{enumerate}[(i)]
		\item if $y\not \in \mathcal{L}$, let $i_0$ be the largest index such that $z_{i_0}\ne 0$ and put $\lambda:=\{z_{i_0}\}$, where $\{\cdot\}$ denotes the distance to the nearest integer.
		\item if $y\in \mathcal{L}$, put $\lambda:=1$.
	\end{enumerate}
	Then, we have 
	\[
	c_1:=\max\limits_{1\le j\le k}\left\{\dfrac{||b_1||}{||b_j^*||}\right\}\qquad\text{and}\qquad	
	\delta:=\lambda\dfrac{||b_1||}{c_1}.
	\]
\end{lemma}

In our application, we are given real numbers $\eta_0,\eta_1,\ldots,\eta_k$ which are linearly independent over $\mathbb{Q}$ and two positive constants $c_3$ and $c_4$ such that 
\begin{align}\label{2.9m}
	|\eta_0+x_1\eta_1+\cdots +x_k \eta_k|\le c_3 \exp(-c_4 H),
\end{align}
where the integers $x_i$ are bounded as $|x_i|\le X_i$ with $X_i$ given upper bounds for $1\le i\le k$. We write $X_0:=\max\limits_{1\le i\le k}\{X_i\}$. The basic idea in such a situation, from \cite{Weg}, is to approximate the linear form \eqref{2.9m} by an approximation lattice. So, we consider the lattice $\mathcal{L}$ generated by the columns of the matrix
$$ \mathcal{A}=\begin{pmatrix}
	1 & 0 &\ldots& 0 & 0 \\
	0 & 1 &\ldots& 0 & 0 \\
	\vdots & \vdots &\vdots& \vdots & \vdots \\
	0 & 0 &\ldots& 1 & 0 \\
	\lfloor C\eta_1\rfloor & \lfloor C\eta_2\rfloor&\ldots & \lfloor C\eta_{k-1}\rfloor& \lfloor C\eta_{k} \rfloor
\end{pmatrix} ,$$
where $C$ is a large constant usually of the size of about $X_0^k$ . Let us assume that we have an LLL--reduced basis $b_1,\ldots, b_k$ of $\mathcal{L}$ and that we have a lower bound $l\left(\mathcal{L},y\right)\ge c_1$ with $y:=(0,0,\ldots,-\lfloor C\eta_0\rfloor)$. Note that $ c_1$ can be computed by using the results of Lemma \ref{lem2.5m}. Then, with these notations the following result  is Lemma VI.1 in \cite{SMA}.
\begin{lemma}[Lemma VI.1 in \cite{SMA}]\label{lem2.6m}
	Let $S:=\displaystyle\sum_{i=1}^{k-1}X_i^2$ and $T:=\dfrac{1+\sum_{i=1}^{k}X_i}{2}$. If $\delta^2\ge T^2+S$, then inequality \eqref{2.9m} implies that we either have $x_1=x_2=\cdots=x_{k-1}=0$ and $x_k=-\dfrac{\lfloor C\eta_0 \rfloor}{\lfloor C\eta_k \rfloor}$, or
	\[
	H\le \dfrac{1}{c_4}\left(\log(Cc_3)-\log\left(\sqrt{\delta^2-S}-T\right)\right).
	\]
\end{lemma}
Python is used to perform all the computations in this work.

\section{Proof of Theorem \ref{thm1.1l}}
\subsection{The case $n\le k+1$}
In this case, $P_n^{(k)} =  F_{2n-1}$, so we combine \eqref{eq:main} with \eqref{eq:2.1} and write
\begin{align}\label{fn}
  F_{2n-1} = \overline{\underbrace{d_1 \ldots d_1}_{\ell \text{ times}} \underbrace{d_2 \ldots d_2}_{m \text{ times}} \underbrace{d_1 \ldots d_1}_{\ell \text{ times}}}.
\end{align}
This Diophantine equation was studied in \cite{batte}, in a more general case. It was proved there that $F_{11}^{(5)}=464$ is the only $k$-Fibonacci number that is a palindromic concatenation of two distinct repdigits. This means that the Diophantine equation \eqref{fn} has no solutions when $n\le k+1$. From now on, we work under the assumption that $n\ge k+2$.

\subsection{The case $n\ge k+2$}
We proceed by proving a series of results. We start with the following.
\begin{lemma}\label{lem:l}
	Let $(\ell, m, n, k)$ be a solution to the Diophantine equation \eqref{eq:main} with $n \ge 9$, $k\ge 2$ and $n\ge k+1$. Then
	$$\ell<4.3\cdot 10^{12}k^5 (\log k)^2\log n.$$
\end{lemma}
\begin{proof}
Let $n$ and $k$ be positive integer solutions to \eqref{eq:main} with $n\ge k+2$ and $k\ge 2$. We rewrite \eqref{eq2.6m} using \eqref{eq2.3l} as
\begin{align*}
	\left|\dfrac{1}{9}\left(d_1\cdot 10^{2\ell+m}-(d_1-d_2)\cdot 10^{\ell+m} +(d_1-d_2)\cdot 10^{\ell}-d_1 \right)-f_k(\alpha)\alpha^{n}\right|<\dfrac{1}{2}.
\end{align*}
Therefore, for all $\ell$, $m\ge 1$, we have
\begin{align*}
	\left|f_k(\alpha)\alpha^{n}-\dfrac{1}{9}\left(d_1\cdot 10^{2\ell+m}\right)\right|
	&<\dfrac{1}{2}+\dfrac{1}{9}\left(|d_1-d_2|\cdot 10^{\ell+m} +|d_1-d_2|\cdot 10^{\ell}+d_1 \right)\\
	&\le \dfrac{1}{2}+ 10^{\ell+m} +10^{\ell}+1 <1.2\cdot 10^{\ell+m}.
\end{align*}
Dividing through both sides of the above inequality by $(d_1\cdot 10^{2\ell+m})/9$, we get
\begin{align}\label{gam1}
	\left|\Gamma_1\right|:=\left|\dfrac{9f_k(\alpha)}{d_1}\cdot 10^{-2\ell-m}\alpha^{n}-1\right|
	&<11\cdot 10^{-\ell}.
\end{align}
Notice that $\Gamma_1\ne 0$, otherwise we would have
\[
d_1\cdot 10^{2\ell + m} = 9 f_k(\alpha) \alpha^{n}.
\]
Let us apply a Galois automorphism from the decomposition field of \(\Psi_k(x)\) over \(\mathbb{Q}\) that sends \(\alpha\) to one of its conjugates \(\alpha_j\) for some \(2 \leq j \leq k\). Taking absolute values after this transformation, we obtain the inequality
\[
10^3 \le d_1 \cdot 10^{2\ell + m} = \left| 9 f_k(\alpha_j) \alpha_j^{n} \right| < 9,
\]
which yields a contradiction. 

The algebraic number field containing the following $\gamma_i$'s is $\mathbb{K} := \mathbb{Q}(\alpha)$ of degree $D = k$. Here, we take $t :=3$,
\begin{equation}\nonumber
	\begin{aligned}
		\gamma_{1}&:=9f_k(\alpha)/d_1,\quad\gamma_{2}:=\alpha,\qquad~~\gamma_{3}:=10,\\
		b_{1}&:=1,\qquad \qquad~~~ b_{2}:=n,\quad\qquad b_{3}:=-2\ell-m.
	\end{aligned}
\end{equation}
Observe that by using inequality \eqref{eqh}, we find
\[
h(\gamma_1) = h\left(9f_k(\alpha)/d_1\right) < \log 9 + 4k \log \varphi + k \log(k+1) + \log 9 < 8k \log k.
\]
Moreover, \( h(\gamma_2) = h(\alpha) < (\log \alpha)/k<(2\log \varphi)/k < 1/k \) and \( h(\gamma_3) = h(10) = \log 10 \). Based on these, we may choose \( A_1 := 8k^2 \log k \), \( A_2 := 1 \), and \( A_3 := k \log 10 \). Furthermore, since \eqref{eq2.4l} asserts that \( 2\ell + m < n \), we can set \( B := n\geq \max\{|b_i|:i=1,2,3\} \).
 Now, by Theorem \ref{thm:Matl},
\begin{align}\label{gam2}
	\log |\Gamma_1| &> -1.4\cdot 30^{6} \cdot 3^{4.5}\cdot k^2 (1+\log k)(1+\log n)\cdot 8k^2\log k\cdot 1\cdot k\log 10\nonumber\\
	&> -9.8\cdot 10^{12}k^5 (\log k)^2\log n.
\end{align}
for all $k\ge 2$ and $n\ge 7$. Comparing \eqref{gam1} and \eqref{gam2}, we get
\begin{align*}
	\ell\log 10-\log 11 &<9.8\cdot 10^{12}k^5 (\log k)^2\log n,
\end{align*}
which gives $\ell<4.3\cdot 10^{12}k^5 (\log k)^2\log n$.
\end{proof}

Next, we prove the following result.
\begin{lemma}\label{lem:m}
	Let $(\ell, m, n, k)$ be a solution to the Diophantine equation \eqref{eq:main} with $n \ge 9$, $k\ge 2$ and $n\ge k+1$. Then
	$$m<5.3\cdot 10^{24}k^9 (\log k)^3(\log n)^2.$$
\end{lemma}
\begin{proof}
Like in the proof of Lemma \ref{lem:l}, we rewrite \eqref{eq2.6m} using \eqref{eq2.3l} as
	\begin{align*}
		\left|\dfrac{1}{9}\left(d_1\cdot 10^{2\ell+m}-(d_1-d_2)\cdot 10^{\ell+m} +(d_1-d_2)\cdot 10^{\ell}-d_1 \right)-f_k(\alpha)\alpha^{n}\right|<\dfrac{1}{2},
	\end{align*}
so that if we consider all $\ell\ge 1$, we have
	\begin{align*}
		\left|f_k(\alpha)\alpha^{n}-\dfrac{1}{9}\left(d_1\cdot 10^{2\ell+m}-(d_1-d_2)\cdot 10^{\ell+m}\right)\right|
		&<\dfrac{1}{2}+\dfrac{1}{9}\left(|d_1-d_2|\cdot 10^{\ell}+d_1 \right)\\
		&<\dfrac{1}{2}+ 10^{\ell}+1 <1.2\cdot 10^{\ell}.
	\end{align*}
Dividing through both sides of the above inequality by $(d_1\cdot 10^{2\ell+m}-(d_1-d_2)\cdot 10^{\ell+m})/9$, we get
	\begin{align}\label{gam3}
	\left|\Gamma_2\right|:=	\left|\dfrac{9 f_k(\alpha)}{(d_1\cdot 10^{\ell}-(d_1-d_2))}\cdot 10^{-\ell-m}\alpha^{n}-1\right|
		&<11\cdot 10^{-m}.
	\end{align}
Again, $\Gamma_2\ne 0$, otherwise we would have
\[
	d_1\cdot 10^{2\ell+m}-(d_1-d_2)\cdot 10^{\ell+m} = 9 f_k(\alpha) \alpha^{n}.
\]
So, if we apply an automorphism from the Galois group of the decomposition field of \(\Psi_k(x)\) over \(\mathbb{Q}\), which maps \(\alpha\) to its conjugate \(\alpha_j\) for \(2 \leq j \leq k\), and then considering the absolute value, we have that for any \(j \geq 2\),
	\[
	10^2 \le  d_1\cdot 10^{2\ell+m}-(d_1-d_2)\cdot 10^{\ell+m} = \left| 9 f_k(\alpha_j) \alpha_j^{n} \right| < 9,
	\]
which is not true. The algebraic number field containing the following $\gamma_i$'s is $\mathbb{K} := \mathbb{Q}(\alpha)$. We have $D = k$, $t :=3$,
	\begin{equation}\nonumber
		\begin{aligned}
			\gamma_{1}&:=9f_k(\alpha)/(d_1\cdot 10^{\ell}-(d_1-d_2)),\quad\gamma_{2}:=\alpha,\qquad~~\gamma_{3}:=10,\\
			b_{1}&:=1,\qquad \qquad\qquad\qquad\qquad\qquad~~~ b_{2}:=n,\qquad\quad b_{3}:=-\ell-m.
		\end{aligned}
	\end{equation}
Notice that 
	\begin{align*}
		h(\gamma_{1})&=h(9f_k(\alpha)/(d_1\cdot 10^{\ell}-(d_1-d_2))
		\le h(9)+h(f_k(\alpha))+h(d_1\cdot 10^{\ell}-(d_1-d_2))\\
		&< \log 9+4k\log \varphi+k\log (k+1)+\log 9+\ell\log 10+2\log 9+\log 2 \\
		&< 5\log 9+8k\log k+\left(4.3\cdot 10^{12}k^5 (\log k)^2\log n\right)\log 10 \\
		&<9.92\cdot 10^{12}k^5 (\log k)^2\log n,
	\end{align*}
	by \eqref{eqh}, so we can take $A_1:=9.92\cdot 10^{12}k^6 (\log k)^2\log n$. As before, $A_2:=1$, $A_3:=k\log 10$, and $B:=n$. 
	
	Now, by Theorem \ref{thm:Matl},
	\begin{align}\label{gam4}
		\log |\Gamma_2| &> -1.4\cdot 30^{6} \cdot 3^{4.5}\cdot k^2 (1+\log k)(1+\log n)\cdot 9.22\cdot 10^{12}k^6 (\log k)^2\log n\cdot 1\cdot k\log 10\nonumber\\
		&> -1.2\cdot 10^{25}k^9 (\log k)^3(\log n)^2.
	\end{align}
	for all $k\ge 2$ and $n\ge 7$. Comparing \eqref{gam3} and \eqref{gam4}, we get
	\begin{align*}
		m\log 10-\log 11 &<1.2\cdot 10^{25}k^9 (\log k)^3(\log n)^2,
	\end{align*}
	which leads to $m<5.3\cdot 10^{24}k^9 (\log k)^3(\log n)^2$.
\end{proof}

At this point, we go back to \eqref{eq2.4m} and use the bounds in Lemmas \ref{lem:l} and \ref{lem:m}, that is
\begin{align*}
	n&<	6(2\ell+m)+2
	<3.2\cdot 10^{25}k^9 (\log k)^3(\log n)^2.
\end{align*}
This gives
\begin{align}\label{boundn}
	n	&<2\cdot 10^{30}k^9 (\log k)^5.
\end{align}

We proceed by distinguishing between two cases on $k$.

\textbf{Case I:} Assuming that $k \leq 1400$, it follows from \eqref{boundn} that  
\[
n	<2\cdot 10^{30}k^9 (\log k)^5 < 8.3 \cdot 10^{63}.
\]
Here, we reduce this upper bound for $n$. To do this, we go back to \eqref{gam1} and recall that
\[
\Gamma_1 := \frac{9}{d_1} \cdot 10^{-2\ell - m} f_k(\alpha) \alpha^{n} - 1 = e^{\Lambda_1} - 1.
\]
Since we already showed that $\Gamma_1 \neq 0$, then $\Lambda_1 \neq 0$. If $\ell \geq 2$, we have $\left| e^{\Lambda_1} - 1 \right| = |\Gamma_1| < 0.5$, which leads to $e^{|\Lambda_1|} \leq 1 + |\Gamma_1| < 1.5$. Thus, 
\[
\left| n \log \alpha - (2\ell + m) \log 10 + \log \left( \frac{9 f_k(\alpha)}{d_1} \right) \right| < \frac{17}{10^\ell}.
\]
We apply the LLL-algorithm for each $k \in [2,1400]$ and $d_1 \in \{1, \dots, 9\}$ to obtain a lower bound for the smallest nonzero value of the above linear form, with integer coefficients bounded in absolute value by $n < 8.3 \cdot 10^{63}$. We consider the lattice  
\[
\mathcal{A} = \begin{pmatrix} 
	1 & 0 & 0 \\ 
	0 & 1 & 0 \\ 
	\lfloor C\log \alpha\rfloor & \lfloor C\log (1/10)\rfloor & \lfloor C\log \left(9 f_k(\alpha)/d_1\right) \rfloor
\end{pmatrix},
\]
with $C := 5.8\cdot 10^{191}$ and $y := (0,0,0)$. Applying Lemma \ref{lem2.5m}, we obtain  
\[
l(\mathcal{L},y) = |\Lambda| > c_1 = 10^{-67} \quad \text{and} \quad \delta = 4.22\cdot 10^{65}.
\]
Via Lemma \ref{lem2.6m}, we have that $S =1.4 \cdot 10^{128}$ and $T = 1.3 \cdot 10^{64}$. Since $\delta^2 \geq T^2 + S$, choosing $c_3 := 17$ and $c_4 := \log 10$, we get $\ell \leq 127$.

Next, we go back to \eqref{gam3}, that is,
$$\Gamma_2:=\dfrac{9 }{(d_1\cdot 10^{\ell}-(d_1-d_2))}\cdot 10^{-\ell-m}f_k(\alpha)\alpha^{n}-1=e^{\Lambda_2}-1,$$
for which we already showed that $\Gamma_2\ne 0$, so $\Lambda_2\ne 0$. Assume for a moment that $m\ge 2$. As before, this implies that 
\begin{align*}
	\left|n \log \alpha - (\ell + m) \log 10 + \log \left( \dfrac{9 f_k(\alpha)}{d_1\cdot 10^{\ell}-(d_1-d_2)} \right)\right|<\dfrac{17}{10^m}.
\end{align*}
For each $k \in [2, 1400]$, $\ell\in[1,124]$, $d_1, d_2\in \{0,\ldots,9\}$, $d_1>0$ and $d_1\ne d_2$, we use the LLL--algorithm to compute a lower bound for the smallest nonzero number of the linear form above with integer coefficients not exceeding $n<8.3\cdot 10^{63}$ in absolute value. So, using the approximation lattice
$$ \mathcal{A}=\begin{pmatrix}
	1 & 0 & 0 \\
	0 & 1 & 0 \\
	\lfloor C\log \alpha\rfloor & \lfloor C\log (1/10)\rfloor& \lfloor C\log \left(9 f_k(\alpha)/(d_1\cdot 10^{\ell}-(d_1-d_2))\right)\rfloor
\end{pmatrix} ,$$
with $C:=  10^{192}$ and $y:=\left(0,0,0\right)$, we apply Lemma \ref{lem2.5m} to get $l\left(\mathcal{L},y\right)=|\Lambda|>c_1=4.7\cdot 10^{-67}$ and $\delta=10^{66}$. However, Lemma \ref{lem2.6m} gives the same values of $S$ and $T$ as before. So choosing $c_3:=17$ and $c_4:=\log 10$, we get $m \le 127$.

Hence, applying inequality \eqref{eq2.4m} yields the upper bound $n \leq 2288$. To complete this case, we implemented a Python-based computational search for all values of $P_n^{(k)}$ with $k \in [2, 1400]$ and $n \in [7, 2288]$ that are palindromic concatenations of two distinct repdigits. This exhaustive check identified only the two cases stated in Theorem~\ref{thm1.1l}; further details are provided in Appendix \ref{app1}.

\textbf{Case II:} If $k > 1400$, then 
\begin{align}\label{boundn2}
	n	<2\cdot 10^{30}k^9 (\log k)^5<\varphi^{k/2}.
\end{align}
We prove the following result.
\begin{lemma}\label{lem:nm}
	Let $(\ell, m, n, k)$ be a solution to the Diophantine equation \eqref{eq:main} with $n \ge 9$, $k>1400$ and $n\ge k+1$. Then
	$$k<4\cdot 10^{34} \qquad\text{and}\qquad n<2\cdot 10^{351}.$$
\end{lemma}
\begin{proof}
Since $k>1400$, then \eqref{boundn2} holds and we can use the sharper estimate in \eqref{pk_b1} together with \eqref{eq2.3l} and write
\begin{align*}
	\bigg|\dfrac{1}{9}\left(d_1\cdot 10^{2\ell+m}-(d_1-d_2)\cdot 10^{\ell+m} +(d_1-d_2)\cdot 10^{\ell}-d_1 \right)&-\dfrac{\varphi^{2n}}{\varphi+2}\bigg|
	=	\left| P_n^{(k)}-\dfrac{\varphi^{2n}}{\varphi+2}\right| \\
	&\le 	\left|P_n^{(k)}- f_k(\alpha)\alpha^{n} \right|+\left|f_k(\alpha)\alpha^{n} -\dfrac{\varphi^{2n}}{\varphi+2}\right| \\
	&< \dfrac{1}{2} + \left(\dfrac{\varphi^{2n}}{\varphi+2}\right)\frac{4}{\varphi^{k/2}}.
\end{align*}
Therefore
\begin{align*}
	\left|\dfrac{1}{9}\cdot d_1\cdot 10^{2\ell+m}-\dfrac{\varphi^{2n}}{\varphi+2}\right|&<\dfrac{1}{2} + \left(\dfrac{\varphi^{2n}}{\varphi+2}\right)\frac{4}{\varphi^{k/2}}+\dfrac{1}{9}\left(|d_1-d_2|\cdot 10^{\ell+m} +|d_1-d_2|\cdot 10^{\ell}+d_1 \right)\\
	&\le \dfrac{1}{2} +\left(\dfrac{\varphi^{2n}}{\varphi+2}\right)\frac{4}{\varphi^{k/2}}+\left( 10^{\ell+m}+10^\ell+1\right)\\
	&< \left(\dfrac{\varphi^{2n}}{\varphi+2}\right)\frac{4}{\varphi^{k/2}}+3\cdot 10^{\ell+m}.
\end{align*}
Now, since $10^{2\ell+m-3}<\alpha^{n-1}<\varphi^{2n-2}$, we have that $10^{\ell+m}<10^{3-\ell}\varphi^{2n-2}$. Thus
\begin{align*}
	\left|\dfrac{1}{9}\cdot d_1\cdot 10^{2\ell+m}-\dfrac{\varphi^{2n}}{\varphi+2}\right|
	&<\left(\dfrac{\varphi^{2n}}{\varphi+2}\right)\frac{4}{\varphi^{k/2}}+3\cdot 10^{3-\ell}\varphi^{2n-2}.
\end{align*}
Dividing both sides by $\varphi^{2n}/(\varphi+2)$ gives
\begin{align*}
	\left|\dfrac{d_1}{9}\cdot 10^{2\ell+m}(\varphi+2)\varphi^{-2n}-1\right| &< \frac{4}{\varphi^{k/2}}+\dfrac{3000\cdot 10^{-\ell}(\varphi+2)}{\varphi^{2}}
	< \frac{4}{\varphi^{k/2}}+\dfrac{4146 }{10^{\ell}}.
\end{align*}
Therefore,
\begin{align}\label{gam5}
|\Gamma_3|:=\left|\dfrac{d_1}{9}\cdot 10^{2\ell+m}(\varphi+2)\varphi^{-2n}-1\right| 	&<\dfrac{4150}{\varphi^{\min\{k/2,\, \ell\log_\varphi 10\}}}.
\end{align}
Clearly $\Gamma_3\ne 0$, otherwise we would have
\[
d_1\cdot 10^{2\ell + m}(\varphi+2) = 9\cdot \varphi^{2n}.
\]
Conjugating the above relation in $\mathbb{Q}(\sqrt{5})$, we would have 
\[
1381<d_1\cdot 10^{2\ell + m}(\bar{\varphi}+2) = 9\cdot \bar{\varphi}^{2n}<9,
\]
which is not true for all $n\ge 7$. The algebraic number field containing the following $\gamma_i$'s is $\mathbb{K} := \mathbb{Q}(\sqrt{5})$. We have $D = 2$, $t :=3$,
\begin{equation}\nonumber
	\begin{aligned}
		\gamma_{1}&:=d_1(\varphi+2)/9,\quad~\gamma_{2}:=10,\qquad\qquad\gamma_{3}:=\varphi,\\
		b_{1}&:=1,\qquad \qquad\qquad b_{2}:=2\ell+m,\qquad b_{3}:=-2n.
	\end{aligned}
\end{equation}
Next,
\begin{align*}
h(\gamma_{1})&\le h(9)+h(d_1)+h(\varphi+2)<4\log 9<9
\end{align*}
$h(\gamma_{2})=h(10)=\log 10$ and $h(\gamma_{3})=h(\varphi)= 0.5\log \varphi $. Therefore, we take $A_1:=18$, $A_2:=2\log 10$ and $A_3:=\log \varphi$. Next, $B :=2n\geq \max\{|b_i|:i=1,2,3\}$ and by Theorem \ref{thm:Matl}, we have
\begin{align}\label{gam6}
	\log |\Gamma_3| &> -1.4\cdot 30^{6} \cdot 3^{4.5}\cdot 2^2 (1+\log 2)(1+\log 2n)\cdot 18\cdot 2\log 10\cdot \log \varphi\nonumber\\
	&> -7.3\cdot 10^{13}\log n.
\end{align}
for all $n\ge 7$. Comparing \eqref{gam5} and \eqref{gam6}, we get
\begin{align*}
	\min\{k/2,\, \ell\log_\varphi 10\}\log \varphi-\log 4150 &<7.3\cdot 10^{13}\log n.
\end{align*}
Next, we investigate two cases.
\begin{enumerate}[(a)]
	\item If $\min\{k/2,\, \ell\log_\varphi 10\}:=k/2$, then $(k/2)\log \varphi-\log 4150 <7.3\cdot 10^{13}\log n$. Therefore,
	\begin{align*}
		k &<3.1\cdot 10^{14}\log n.
	\end{align*}
By \eqref{boundn2}, $n	<2\cdot 10^{30}k^9 (\log k)^5$ and so $\log n<25\log k$, for all $k> 1400$. Therefore, we have $$k<7.8\cdot 10^{15}\log k,$$
which gives $k<5.8\cdot 10^{17}$.
	
	\item If $\min\{k/2,\, \ell\log_\varphi 10\}:=\ell\log_\varphi 10$, then $(\ell\log_\varphi 10)\log \varphi-\log 4150 <7.3\cdot 10^{13}\log n$, and so
	\begin{align*}
		\ell &<3.2\cdot 10^{13}\log n.
	\end{align*}
We proceed as in \eqref{gam5} via
\begin{align*}
	\left|\dfrac{1}{9}\left(d_1\cdot 10^{2\ell+m}-(d_1-d_2)\cdot 10^{\ell+m}  \right)-\dfrac{\varphi^{2n}}{\varphi+2}\right|
	&< \dfrac{1}{2} + \left(\dfrac{\varphi^{2n}}{\varphi+2}\right)\frac{4}{\varphi^{k/2}}+\dfrac{1}{9}\left(|d_1-d_1|\cdot 10^{\ell}+d_1\right).
\end{align*}
Therefore
\begin{align*}
	\left|\dfrac{1}{9}\left(d_1\cdot 10^{\ell}-(d_1-d_2)\right)10^{\ell+m}-\dfrac{\varphi^{2n}}{\varphi+2}\right|&< \dfrac{1}{2} +\left(\dfrac{\varphi^{2n}}{\varphi+2}\right)\frac{4}{\varphi^{k/2}}+\left( 10^\ell+1\right)\\
	&< \left(\dfrac{\varphi^{2n}}{\varphi+2}\right)\frac{4}{\varphi^{k/2}}+2\cdot 10^{\ell}.
\end{align*}
Dividing both sides by $\varphi^{2n}/(\varphi+2)$, we get that
\begin{align*}
	\left|\dfrac{1}{9}\left(d_1\cdot 10^{\ell}-(d_1-d_2)\right)10^{\ell+m}(\varphi+2)\varphi^{-2n}-1\right| &
	< \frac{4}{\varphi^{k/2}}+\dfrac{2\cdot 10^{\ell} }{\varphi^{2n}}.
\end{align*}
Now, notice that since $\ell\log_\varphi 10\le k/2$, (i.e., $10^\ell\le \varphi^{k/2}$) and $n\ge k+2>k/2$, then
\begin{align}\label{gam7}
	|\Gamma_4|:=\left|\dfrac{1}{9}\left(d_1\cdot 10^{\ell}-(d_1-d_2)\right)10^{\ell+m}(\varphi+2)\varphi^{-2n}-1\right|	&<\dfrac{6}{\varphi^{k/2}}.
\end{align}
Clearly $\Gamma_4\ne 0$, otherwise we would have
\[
\left(d_1\cdot 10^{\ell}-(d_1-d_2)\right)10^{\ell+m}(\varphi+2)= 9\cdot \varphi^{2n}.
\]
Conjugating the above relation in $\mathbb{Q}(\sqrt{5})$, we would have 
\[
361<\left(d_1\cdot 10^{\ell}-(d_1-d_2)\right)10^{\ell+m}(\bar{\varphi}+2) = 9\cdot \bar{\varphi}^{2n}<9,
\]
which is false for all $n\ge 7$. The algebraic number field containing the following $\gamma_i$'s is $\mathbb{K} := \mathbb{Q}(\sqrt{5})$. We have $D = 2$, $t :=3$,
\begin{equation}\nonumber
	\begin{aligned}
		\gamma_{1}&:=\left(d_1\cdot 10^{\ell}-(d_1-d_2)\right)(\varphi+2)/9,\quad~\gamma_{2}:=10,\qquad\qquad\gamma_{3}:=\varphi,\\
		b_{1}&:=1,\qquad \qquad\qquad\qquad\qquad\qquad\qquad~ b_{2}:=\ell+m,\qquad~~ b_{3}:=-2n.
	\end{aligned}
\end{equation}
As before, $A_2:=2\log 10$, $A_3:=\log \varphi$ and $B:=2n$. For $A_1$, we first find
\begin{align*}
	h(\gamma_{1})&=h(\left(d_1\cdot 10^{\ell}-(d_1-d_2)\right)(\varphi+2)/9)\\
	&\le h(d_1)+h(10^\ell)+h(d_1-d_2)+h(\varphi+2)+h(9)+2\log 2\\
&<7\log 9 +\ell\log 10+0.5	\log (\varphi+2)\\
&<7\log 9 +3.2\cdot 10^{13}\log n\log 10+0.5	\log (\varphi+2)\\
&<7.4\cdot 10^{13}\log n,
\end{align*}
Therefore, we take $A_1:=1.5\cdot 10^{14}\log n$. By Theorem \ref{thm:Matl}, we have
\begin{align}\label{gam8}
	\log |\Gamma_4| &> -1.4\cdot 30^{6} \cdot 3^{4.5}\cdot 2^2 (1+\log 2)(1+\log 2n)\cdot 1.5\cdot 10^{14}\log n\cdot 2\log 10\cdot \log \varphi\nonumber\\
	&> -6.2\cdot 10^{26}(\log n)^2.
\end{align}
for all $n\ge 7$. Comparing \eqref{gam7} and \eqref{gam8}, we get
\begin{align*}
	(k/2)\log \varphi-\log 6 &<6.2\cdot 10^{26}(\log n)^2.
\end{align*}
so that $$k<2.6\cdot 10^{27}(\log n)^2.$$
	
Since $n	<2\cdot 10^{30}k^9 (\log k)^5$ and so $\log n<25\log k$, for all $k> 1400$, we have 
$$k<2\cdot 10^{30}(\log k)^2,$$
which gives $k<4\cdot 10^{34}$ and thus $n<2\cdot 10^{351}$.
\end{enumerate}
This completes the proof.
\end{proof}

We reduce the bounds in Lemma \ref{lem:nm}. To do this, we revisit \eqref{gam5} and recall that
\[
\Gamma_3 := \dfrac{d_1}{9}\cdot 10^{2\ell+m}(\varphi+2)\varphi^{-2n}-1.
\]
Since we showed that $\Gamma_3 \neq 0$, then $\Lambda_3 \neq 0$ and we obtain the inequality  
\[
\left| e^{\Lambda_3} - 1 \right| = |\Gamma_3| < 0.5,
\]
which leads to $e^{|\Lambda_3|} \leq 1 + |\Gamma_3| < 1.5$. Therefore 
\[
\left|  \log \left(\dfrac{d_1(\varphi+2)}{9}\right) + (2\ell + m) \log 10 -n \log (\varphi^2)  \right| < \dfrac{6225}{\varphi^{\min\{k/2,\, \ell\log_\varphi 10\}}}.
\]
For each $d_1\in\{1,\ldots , 9\}$, we apply the LLL-algorithm to obtain a lower bound for the smallest nonzero value of the above linear form, bounded by integer coefficients with absolute values at most $2 \cdot 10^{351}$. Here, we consider the lattice  
\[
\mathcal{A}_1 = \begin{pmatrix} 
	1 & 0 & 0 \\ 
	0 & 1 & 0 \\ 
	\lfloor C\log (d_1(\varphi+2)/9)\rfloor & \lfloor C\log 10\rfloor & \lfloor C\log (1/\varphi^2) \rfloor
\end{pmatrix},
\]
where we set $C := 10^{1055}$ and $y := (0,0,0)$. Applying Lemma \ref{lem2.5m}, we obtain  
\[
l(\mathcal{L},y) = |\Lambda| > c_1 = 10^{-355} \quad \text{and} \quad \delta = 1.096\cdot 10^{352}.
\]
Using Lemma \ref{lem2.6m}, we conclude that $S = 1.2 \cdot 10^{703}$ and $T = 3.1 \cdot 10^{351}$. Since $\delta^2 \geq T^2 + S$, then choosing $c_3 := 6225$ and $c_4 := \log \varphi$, we get  $\min\{k/2,\, \ell\log_\varphi 10\} \leq 3382$. 

We proceed the analysis in two ways.
\begin{enumerate}[(a)]
	\item If $\min\{k/2,\, \ell\log_\varphi 10\}:=k/2$, then $k/2 \leq 3382$, or $k\le 6764$.

	\item If $\min\{k/2,\, \ell\log_\varphi 10\}:=\ell\log_\varphi 10$, then $\ell\log_\varphi 10\le 3382$. Therefore $\ell \le 706$.
	We revisit equation \eqref{gam7} and recall the expression
	\[
\Gamma_4:=\dfrac{1}{9}\left(d_1\cdot 10^{\ell}-(d_1-d_2)\right)10^{\ell+m}(\varphi+2)\varphi^{-2n}-1.
	\]
	Since we have already shown that \(\Gamma_4 \neq 0\), it follows that \(\Lambda_4 \neq 0\). So, we arrive at
	\[
	\left|  \log \left(\frac{\left(d_1\cdot 10^{\ell}-(d_1-d_2)\right)(\varphi+2)}{9}\right) + (\ell + m) \log 10 -n \log (\varphi^2)  \right| < \frac{9}{\varphi^{k/2}}.
	\]
For each $\ell\in[1,706]$, $d_1, d_2\in \{0,\ldots,9\}$, $d_1>0$ and $d_1\ne d_2$, we use the LLL--algorithm to compute a lower bound for the smallest nonzero number of the linear form above with integer coefficients not exceeding $n<2\cdot 10^{351}$ in absolute value. So, we consider the lattice
	\[
	\mathcal{A}_2 = \begin{pmatrix} 
		1 & 0 & 0 \\ 
		0 & 1 & 0 \\ 
		\lfloor C\log \left(\left(d_1\cdot 10^{\ell}-(d_1-d_2)\right)(\varphi+2)/9\right)\rfloor & \lfloor C\log 10\rfloor & \lfloor C\log (1/\varphi^2) \rfloor
	\end{pmatrix},
	\]
	where we define \(C := 10^{1054}\) and set \(y := (0,0,0)\) as before. With the same values for $\delta$, $S$ and $T$ as before, we choose \(c_3 := 9\) and \(c_4 := \log \varphi\). It follows that $k/2 \leq 3364$ from which we get $k\le 6728$.
	
\end{enumerate}
Therefore, we always have $k\le 6764$. 

The above restriction on $k$ implies, through inequality \eqref{boundn2}, that $n < 10^{70}$. We therefore proceed with a second round of reduction using this updated bound on $n$. Going back to equation \eqref{gam5}, as previously done, we use the LLL-algorithm to estimate a lower bound for the smallest nonzero value of the corresponding linear form, under the assumption that the involved integer coefficients are bounded above by $n < 10^{70}$. We use the same lattice  $\mathcal{A}_1$ as before but now $C := 10^{211}$ and $y := (0,0,0)$. By applying Lemma \ref{lem2.5m}, we get  
\[
l(\mathcal{L},y) = |\Lambda| > c_1 = 10^{-72} \quad \text{and} \quad \delta = 4.1\cdot 10^{71}.
\]
Using Lemma \ref{lem2.6m}, we conclude that $S = 3 \cdot 10^{140}$ and $T = 1.51 \cdot 10^{70}$. Since $\delta^2 \geq T^2 + S$, and choosing $c_3 := 6225$ and $c_4 := \log \varphi$, it follows that $\min\{k/2,\, \ell\log_\varphi 10\} \leq 685$. 

We again analyze two cases:
\begin{enumerate}[(a)]
	\item If $\min\{k/2,\, \ell\log_\varphi 10\} := k/2$, then $k/2 \leq 685$, which implies $k \leq 1370$. This contradicts  the working assumption $k>1400$.
	
	\item If $\min\{k/2,\, \ell\log_\varphi 10\} := \ell\log_\varphi 10$, then $\ell\log_\varphi 10 \leq 685$, giving $\ell \leq 143$.
	
	Revisiting equation \eqref{gam7}, we again apply the LLL-algorithm, with integer coefficients restricted to absolute values below $n <  10^{70}$. By the same lattice $\mathcal{A}_2$ as before, we have $C := 10^{211}$ and $y := (0,0,0)$. Using Lemma \ref{lem2.5m} and Lemma \ref{lem2.6m}, we have $\delta=2.2\cdot 10^{71}$, $S = 3 \cdot 10^{140}$ and $T = 1.51 \cdot 10^{70}$. Selecting $c_3 := 9$ and $c_4 := \log \varphi$, we deduce that $k/2 \leq 672$, leading to $k \leq 1344$. This again contradicts  the working assumption $k>1400$.
	
\end{enumerate}
Therefore, the proof is complete. \qed

\section*{Acknowledgments} 
The first author was funded by the cost centre 0730 of the Mathematics Division at Stellenbosch University.

\section*{Addresses}

$ ^{1} $ Mathematics Division, Stellenbosch University, Stellenbosch, South Africa.

Email: \url{hbatte91@gmail.com}\\
$ ^{2} $ Department of Mathematics, Makerere University, Kampala, Uganda.

Email: \url{dariusguma2@gmail.com}

\newpage
\appendix
\section{Appendices}

\subsection{Py Code I}\label{app1}
\begin{verbatim}
def k_pell_sequence(k, n_max):
    """Generates the k-Pell sequence up to n_max terms."""
    P = [0] * (k - 1) + [1]  # P_{2-k} to P_1
    for _ in range(n_max - len(P) + 1):
       term = 2 * P[-1] + sum(P[-2:-k-1:-1])  # 2*P_{n-1} + P_{n-2} + ... + P_{n-k}
       P.append(term)
    return P
def is_palindromic_concatenation(num):
    s = str(num)
    n = len(s)
    for l in range(1, n // 2 + 1):
        m = n - 2 * l
        if m <= 0:
           continue
        part1 = s[:l]
        part2 = s[l:l + m]
        part3 = s[l + m:]
        if (
           part1 == part3
           and len(set(part1)) == 1
           and len(set(part2)) == 1
           and part1 != part2
        ):
           return True
        return False
def find_all_palindromic_k_pells(k_min=2, k_max=1400, n_max=2288):
    results = []
    for k in range(k_min, k_max + 1):
        pells = k_pell_sequence(k, n_max)
        for i, val in enumerate(pells):
            math_n = i + 2 - k  # Convert index to actual P_n^{(k)}
            if math_n < 2:
                continue
            if val > 0 and is_palindromic_concatenation(val):
                results.append((k, math_n, val))
     return results
pell_results = find_all_palindromic_k_pells()
print("\nFound palindromic repdigit k-Pell numbers:")
for k, n, value in pell_results:
    print(f"P_{n}^({k}) = {value}")
\end{verbatim}


\begin{thebibliography}{99}
	
	\bibitem{BW}
	Baker, A., \& Wüstholz, G.: Logarithmic forms and group varieties, 19--62  (1993).
	
	\bibitem{banks}
	Banks, W.D., \& Luca, F.: Concatenations with binary recurrent sequences. \textit{J. Integer Seq.} \textbf{8}(5), Art. 05.1.3 (2005).
	
	\bibitem{batte2}
	Batte, H.: Lucas numbers that are palindromic concatenations of two distinct repdigits. \textit{Mathematica Pannonica New Series}, (2025). DOI: \url{https://doi.org/10.1556/314.2025.00003}.
	
	\bibitem{batte}
	Batte, H., \& Luca, F.: $k$-Fibonacci numbers that are palindromic concatenations of two distinct Repdigits. \url{https://doi.org/10.48550/arXiv.2504.10138}
	
	\bibitem{kaggwa}
	Batte, H., \& Kaggwa, P.: Perrin numbers that are palindromic concatenations of two repdigits,  (2024).  \url{https://doi.org/10.5281/zenodo.13886753}.	
	
	\bibitem{BraHer}
	Bravo, J. J., \& Herrera, J. L.: Repdigits in generalized Pell sequences. \textit{Arch. Math.}	\textbf{56}(4), 249--262 (2020).
	
	\bibitem{Herera}
	Bravo, J. J., Herrera, J. L., \& Luca, F.: On a generalization of the Pell sequence. \textit{Mathematica Bohemica}, 146(2), 199--213  (2021).
	
	\bibitem{matl}
	Bugeaud Y.,  Maurice M., \& Siksek S.: Classical and modular approaches to exponential Diophantine equations I. Fibonacci and Lucas perfect powers.  \textit{Annals of Mathematics}, {\bf 163}: 969--1018 (2006).
	
	\bibitem{Cas}
	Cassels, J. W. S.: An introduction to the geometry of numbers. Springer Science \& Business Media (2012).
	
	\bibitem{chal}
	Chalebgwa, T. P., \& Ddamulira, M.: Padovan numbers which are palindromic concatenations of two distinct repdigits. \textit{Revista de la Real Academia de Ciencias Exactas, Físicas y Naturales. Serie A. Matemáticas}, \textbf{115}(3), 108 (2021).
	
	\bibitem{emong}
	Ddamulira, M., Emong, P., \& Mirumbe, G. I.: Palindromic Concatenations of Two Distinct Repdigits in Narayana's Cows Sequence. \textit{Bulletin of the Iranian Mathematical Society}, \textbf{50}(3), 35 (2024).
	
	\bibitem{Weg}
	de Weger, B. M.: Solving exponential Diophantine equations using lattice basis reduction algorithms, {\it Journal of Number Theory\/} {\bf 26},  325--367 (1987).	
	
	\bibitem{faye}
	Faye, B., Luca, F.: Pell and Pell-Lucas numbers with only one distinct digit.	\textit{Ann. Math.} \textbf{45}, 55--60 (2015)
	
	\bibitem{kilis}
	Kili\c c, E.: On the usual Fibonacci and generalized order-$k$ Pell numbers. \textit{Ars Combin.} \textbf{109}, 391--403 (2013).
	
	\bibitem{LLL} 
	Lenstra, A. K., Lenstra, H. W., \& Lovász, L.: Factoring polynomials with rational coefficients. {\it Mathematische Annalen\/} {\bf 261}, 515--534 (1982).
	
	\bibitem{SMA}
	Smart, N. P.: The algorithmic resolution of Diophantine equations: a computational cookbook (Vol. 41). Cambridge University Press (1998).
		
\end{thebibliography}
\end{document}